\documentclass[11pt]{article}

\usepackage{amsmath,amsfonts,xspace,amsthm,amssymb,mathtools,mathrsfs,color}
\newtheorem{definition}{Definition}
\newtheorem{remark}{Remark}
\newtheorem{proposition}{Proposition}
\newtheorem*{theorem}{Theorem}
\newtheorem*{example}{Example}

\begin{document}

\title{Weyl Prior and Bayesian Statistics}
\author{Ruichao Jiang$^{1}$, Javad Tavakoli$^{1,*}$ and Yiqiang Zhao$^{2}$ \\
$^{1}$ Department of Mathematics \\ The University of British Columbia Okanagan \\  Kelowna, BC, Canada  V1V 1V7 \\
$^{2}$ School of Mathematics and Statistics\\ Carlton University \\ Ottawa, ON, Canada K1S 5B6 \\
$^{*}$ Correspondence: javad.tavakoli@ubc.ca
}
\date{April 2020}

\maketitle

\begin{abstract}

When using Bayesian inference, one needs to choose a prior distribution for parameters. The well-known Jeffreys prior is based on the Riemann metric tensor on a statistical manifold. Takeuchi and Amari defined the $\alpha$-parallel prior,, which generalized the Jeffreys prior by exploiting higher-order geometric object, known as Chentsov-Amari tensor. In this paper, we propose a new prior based on the Weyl structure on a statistical manifold. It turns out that our prior is a special case of the $\alpha$-parallel prior with the parameter $\alpha$ equals $-n$, where $n$ is the dimension of the underlying statistical manifold and the minus sign is a result of conventions used in the definition of $\alpha$-connections. This makes the choice for the parameter $\alpha$ more canonical. We calculated the Weyl prior for univariate Gaussian and multivariate Gaussian distribution. The Weyl prior of the univariate Gaussian turns out to be the uniform prior.
\medskip

\end{abstract}

\section{Introduction}
In Bayesian inference, a parameter is regarded as a random variable $\Theta$. A density of $\Theta$ is called, by abuse of terminology, a prior distribution $p(\theta)$. After collecting some data, one obtains a conditional density $p(x|\theta)$, referred to as the likelihood function. The Bayes' theorem then computes the posterior distribution $p(\theta|x)$ using $p(\theta)$ and $p(x|\theta)$.  This is interpreted as an update of the information about the unknown parameter $\Theta$ in Bayesian inference. One such choice of the prior distribution $p(\theta)$ is Jeffreys prior $^J\omega$, which is the correct choice of uniform distribution. Here the word "uniform" means uninformative, unfavorable of any particular choice of the parameter.

Information geometry, in its narrowest sense, is an attempt to use differential geometry to study statistical inference. It has found applications in statistical inference, signal processing, and machine learning \cite{Amari2016}. \cite{calin,nielson} are two elementary introductions. In information geometry, geometric structures, for example metric tensors $g$ and affine connections $\nabla$, can be put on the set of prior distributions $\mathcal{P}(\theta)$. These geometric structures help to single out some particular prior distributions, for example, the Jeffreys prior $^J\omega,$ which, by the fundamental theorem of Riemannian geometry, is the unique volume form parallel with respect to the Levi-Civita connection $^{LC}\nabla$. Since the Jeffreys prior $^J\omega$ is provided by geometry, it is automatically invariant under reparametrization, which reflects the opinion that information can be at best not lost during a transformation of parameters and this is encoded in the notion of sufficient statistics. Similarly, if one can find a unique prior distribution satisfying some specified geometric conditions, then that prior distribution is called canonically chosen.  Matsuzoe, Takeuchi, and Amari used information geometry to define the $\alpha$-parallel prior $^\alpha\omega$ such that when $\alpha=0$, it reduces to the Jeffreys prior \cite{matsuzoe2006}.

Historically, Weyl proposed a generalization of general relativity to unify gravity and electromagnetism. Einstein soon pointed out that Weyl's theory predicted substantial broadening of the characteristic length of atoms, which is contradictory to the well-observed thin atomic spectra. Even though Weyl geometry failed the unification of gravity and electromagnetism, which is yet an open problem, Weyl geometry has found applications in possible generalization of general relativity \cite{ciambelli2019weyl} and the differential-geometric study of defects in continuum mechanics \cite{yavari}. Weyl geometry is kept in mathematics. The relation of affine differential geometry, Weyl geometry, and Riemannian geometry are shown below. Let $\pi:E\to{B}$ be a fibre bundle with base $B$ and let each fibre $\pi^{-1}(x),$ $x\in{B}$, be a Lie group $G$, called the structure group of the fibre bundle. For different $G$ we obtain different geometries as follows.
\begin{enumerate}
    \item $GL(n),$ the general linear group, affine differential geometry;
    \item $C(n)\coloneqq\{kA\ |\ k\in{\mathbb{R}^{+}},\ A\in{O(n)}\},$ the conformal group, Weyl geometry;
    \item $O(n),$ the orthogonal group, Riemannian geometry.
\end{enumerate}
By the reduction of structure groups $O(n)\subseteq C(n)\subseteq GL(n)$, and the fact that the smaller the structure group is, the more geometric properties are expected. In our case, the reduction of group gives rise to a canonical choice for the parameter $\alpha$ of the $\alpha$-parallel prior. For more about bundle-theoretic differential geometry, see \cite{kobayashi}.

In this paper, we will use Weyl geometry to define a prior distribution for Bayesian inference, which we call the Weyl prior. We will elucidate the relation between the dimension of a statistical manifold and the parameter $\alpha$ in Takeuchi and Amari's $\alpha$-parallel prior.

The organization of the paper is as follows: In Section~2, we review information geometry and the $\alpha$-parallel prior. We discuss Weyl geometry in Section~3. We define the Weyl prior, and elucidate the relation between the Weyl prior and the $\alpha$-parallel prior in Section~4. We calculate the Weyl prior for the univariate Gaussian distribution as an example in Section~5 and the multivariate Gaussian distribution in Section~6. All functions in the paper are real-valued and smooth, all connections are torsion-free, and the Einstein summation rule is used.

\section{Information geometry and $\alpha-$priors}
In this section, we review some basics of information geometry. For more details, see \cite{Amari2016}.

Let us consider a statistical model $\mathcal{P},$ which is a set of parametric densities $\mathcal{P}=\{p(x|\theta)\}.$ $\mathcal{P}$ can be geometrized as follows: first, we introduce the Fisher metric tensor, which is a 2nd order tensor,
\begin{definition}
The Fisher metric tensor is defined by
\begin{equation}
    g_{ij}=\text{E}_{\theta}\left[\partial_{i}l\partial_{j}l\right]
\end{equation}
where $E_{\theta}$ is the the transition kernel $X\times\Theta\to{[0,\infty)}$, $l$ is the log likelihood function, and $\partial_{i}$ is the partial derivative with respect to coordinate $i$.
\end{definition}
Then, we introduce the Amari-Chentsov tensor, which is a 3rd order tensor.
\begin{definition}[Amari-Chentsov Tensor]
The Amari-Chentsov tensor $C$ is defined by
\begin{equation}
C_{ijk}=\text{E}_{\theta}\left[\partial_{i}l\partial_{j}l\partial_{k}l\right].
\end{equation}
\end{definition}
\begin{remark}
The Amari-Chentsov tensor $C$ defined above satisfies
\[
    C = \nabla{g}.
\]
In other words, $C$ is the covariant derivative of the metric tensor $g.$
\end{remark}
In Riemannian geometry, $C$ vanishes everywhere, which is required if the length of a tangent vector is to be preserved under the parallel transport. In information geometry, this requirement is dropped and thus a duality theory arises. Let $\nabla$ be an arbitrary torsion-free affine connection on a Riemannian manifold $(M,g).$ The dual connection $\nabla^*$ of $\nabla$ plays an important role in information geometry.
\begin{definition}[Dual Connection]
The dual connection $\nabla^*$ on a Riemannian manifold $(M,g)$ with affine connection $\nabla$ is defined as the unique affine connection satisfying the following equation:
\begin{equation}
    Xg_p(Y,Z)=g_p(\nabla_XY,Z)+g_p(Y,\nabla^{*}_{X}Z),
\end{equation}
where $p\in{M}$ and $X,Y,Z\in{T_{p}M.}$
\end{definition}
\begin{remark}
The dual connection $\nabla^*$ preserves the metric tensor $g$ together with $\nabla:$
$$
g_p(X,Y)=g_q(\Pi{X},\Pi^*Y),
$$
where $X,Y\in{T_{p}M,}$ and $\Pi$ and $\Pi^*$ are parallel transports induced by $\nabla$ and $\nabla^*$, respectively, along some curve from $p$ to $q$. In general, $g_p(X,Y)\neq g_q(\Pi{X},\Pi{Y})$ and $g_p(X,Y)\neq g_q(\Pi^*{X},\Pi^*Y)$, unless $\nabla=\nabla^*=^{LC}\nabla$. See \cite{Amari2016}.
\end{remark}
Now we introduce $\alpha$-connections.
\begin{definition}
\label{def}
The $\alpha$-connections are defined in terms of Christoffel symbols by
\begin{equation}
\label{alphaConnection}
    ^\alpha\Gamma^i_{jk} = ^{LC}\Gamma^i_{jk} - \frac{\alpha}{2}g^{il}C_{ljk},
\end{equation}
where $\alpha\in{\mathbb{R}}$ and LC stands for Levi-Civita.
\end{definition}
\begin{remark}
\label{dual}
The dual connection of $\nabla^{\alpha}$ is then given by
\begin{equation}
    \nabla^{\alpha*}=\nabla^{-\alpha}.
\end{equation}
\end{remark}
\begin{remark}
\label{comparison}
The $\alpha$-parallel prior $^{\alpha}\omega$ is the volume form parallel with respect to $\nabla^{\alpha}.$ Unlike the Jeffreys prior, which always exists, the $\alpha$-parallel prior do not necessarily exist. $\alpha$-parallel prior exists if and only if the Ricci curvature tensor is symmetric \cite{matsuzoe2006}. However, if $^{\alpha}\omega$ exists for one $\alpha\in{\mathbb{R}},$ then it exists for all $\alpha$ \cite{takeuchi}.
\end{remark}
The following characterization will be used in Section~5 to obtain the relation between the $\alpha$-parallel prior and the Weyl prior defined therein.
\begin{proposition}\cite{matsuzoe2006}
\label{alpha}
Let $(M,g,^\alpha\nabla)$ be a statistical manifold. If there exists an exact $1$-form $T=\text{d}\Omega$ for some function $\Omega$ determined by $\nabla$ and $g$, then the $\alpha$-parallel prior is $^{\alpha}\omega=\exp\{-\frac{\alpha}{2}\Omega\}\sqrt{\det{g}}.$
\end{proposition}
\begin{remark}
$\text{d}\Omega$ is known as the Chebyshev $1$-form. A differential form $\phi$ is called closed if the exterior derivative vanishes i.e. $\text{d}\phi=0$, and is called exact if there exists a differential form $\varphi$ such that $\phi=\text{d}\varphi$. By definition, every exact form is closed. By Poincare's lemma, every closed form is locally exact. Because statistical manifolds are simply connected, closedness implies exactness.
\end{remark}
\section{Weyl Geometry}
In this section, we review some concepts of Weyl geometry which are needed in the next section. For more details, see \cite{folland}.

Two Riemannian metrics $g$ and $g'$ on a manifold $M$ are said to be conformally equivalent if $g'=\mathrm e^\lambda g$ for some smooth function $\lambda$ on $M$.

A conformal structure $\mathcal{C}$ on $M$ is an equivalent class of conformally equivalent Riemannian metrics, i.e. $\mathcal{C}\coloneqq\{g' | g'=\mathrm e^\lambda g\}$.

A Weyl structure is a map $F: \mathcal{C} \to\Lambda^1(M)$ from the conformal structure $\mathcal{C}$ to the set of $1-$forms on $M,$ satisfying
$$
F(\mathrm e^\lambda g)=F(g)-d\lambda.
$$
The image of $g$ under $F$ is called the Weyl $1$-form $F(g)\coloneqq\varphi.$

A Weyl structure enables us to translate a scalar product $(\ ,\ )_p$ at $p$ to $(\ ,\ )_q$ at $q$ along a curve $c: [0,1]\to M:$
\begin{equation}
    (\ ,\ )_q=\exp{\left[\int_0^1c^*\varphi\right]}g_q,
\end{equation}
where $c^*\varphi$ is the pullback of the Weyl $1-$form $\varphi$ along curve $c.$ A Weyl manifold is a manifold with a Weyl structure.
\begin{remark}
The meaning of this equation is: If we start with a scalar product $(,)_p$ at a point $p$ arising from the conformal class $\mathcal{C}$, then there exists a metric tensor $g\in \mathcal{C}$ 
extending $(,)_p$, i.e., $g_p=(,)_p$. The value of this particular choice of $g$ at another point $q$ is $g_q$. However, different choice of $g$ gives rise to different $g_q$. The scalar product $(,)_q$ determined by Weyl translation is proven to be independent of $g$ \cite{folland}. Hence by Weyl translation, we can compare lengths of vectors at different points on a Weyl manifold; whereas with only the conformal structure $\mathcal{C}$, we can only compare ratios of lengths.
\end{remark}

An affine connection $\nabla$ is said to be a Weyl connection if the parallel transport of a scalar product under $\nabla$ coincides with the Weyl translation.

The Weyl connection is characterized by the following propositions.
\begin{proposition}[\cite{folland}]
An affine connection $\nabla$ is a Weyl connection if and only if $\nabla g +\varphi\otimes g=0$ for all $g\in\mathcal{C}.$
\end{proposition}
\begin{proposition}[Fundamental Theorem of Weyl Geometry \cite{folland}]
\label{weylunique}
There exists a unique torsion-free Weyl connection $^{W}\nabla$ on a Weyl manifold $M.$ The Christoffel symbols of $^{W}\nabla$ are given by
\begin{equation}
\label{tobecontacted}
    ^W\Gamma^i_{jk}=^{LC}\Gamma^i_{jk}+\frac{1}{2}\left(\delta^i_j\varphi_{k}+\delta^i_k\varphi_{j}-g^{im}g_{jk}\varphi_{m}\right),
\end{equation}
where $\delta^i_j$ is the Kronecker delta.
\end{proposition}
\section{Weyl Prior}
In this section, we define the Weyl prior and show its relation to the $\alpha$-parallel prior.

First, we define the Weyl prior as follows.
\begin{definition}[Weyl Prior]
Let $(M,g)$ be an $n-$dimensional Riemannian manifold with the conformal structure $\mathcal{C}=[g]$ and the Weyl structure $F.$ Let $^{W}\nabla$ be the Weyl connection. The Weyl prior $^{W}\omega$ is  defined as the unique volume form parallel with respect to $^{W}\nabla.$
\end{definition}
\begin{remark}
The uniqueness of the Weyl prior is the result of the uniqueness of the Weyl connection.
\end{remark}
Now we prove the main result of this paper.
\begin{theorem}
Let $(M,g)$ be a Riemannian manifold. Let $^W\nabla$ and $^{-n}\nabla$ be the Weyl connection and the $-n$-connection, i.e. the $\alpha$-connection with $\alpha=-n$, where $n$ is the dimension of $M.$ Suppose that the $-n$-prior $^{-n}\omega$ exists, then
$$
^W\omega=^{-n}\omega.
$$
\end{theorem}
\begin{proof}
Consider an arbitrary volume form $f\sqrt{\det{g}},$ where $f$ is a positive function on $M.$ For $f\sqrt{\det{g}}$ to be parallel with respect to $^W\omega,$ it is necessary and sufficient that
\begin{align}
\label{condition}
^W\nabla\left(f\sqrt{\det{g}}\right)=f ^W\nabla\sqrt{\det{g}}+^W\nabla f\sqrt{\det{g}}=0.
\end{align}
Componentwise, Equation (\ref{condition}) becomes
\begin{align}
\label{condtioncomponent}
f ^W\nabla_{j}\sqrt{\det{g}}+^W\nabla_{j}f\sqrt{\det{g}}=f ^W\nabla_{j}\sqrt{\det{g}}+\partial_{j}f\sqrt{\det{g}},
\end{align}
since covariant derivative coincides with partial derivative for functions.

Since $\sqrt{\det{g}}$ is a scalar density of weight $1,$ its covariant derivative is given by
\begin{equation}
\label{cal}
^W\nabla_{j}\sqrt{\det{g}}=\partial_{j}\sqrt{\det{g}}-^W\Gamma_{j}\sqrt{\det{g}},
\end{equation}
where $\Gamma_{j}$ is obtained by the contraction of Equation (\ref{tobecontacted}) over $i$ and $k:$
\begin{align*}
^W\Gamma_{j}=^W\Gamma^i_{ji}&=^{LC}\Gamma^i_{ji}+\frac{1}{2}\left(\delta^i_j\varphi_i+\delta^i_i\varphi_j-g^{im}g_{ji}\varphi_m\right)\\
&=\partial_j\ln{\sqrt{\det{g}}}+\frac{1}{2}\left(\varphi_j+n\varphi_j-\delta^m_j\varphi_m\right)\\
&=\partial_j\ln{\sqrt{\det{g}}}+\frac{n}{2}\varphi_j.
\end{align*}
Substitute Equation (\ref{cal}) into Equation (\ref{condtioncomponent}), we obtain
\begin{equation}
\label{collect}
    \partial_jf=\frac{n}{2}\varphi_jf.
\end{equation}
Since the covariant derivative coincides with exterior derivative for functions, collect indices in Equation (\ref{collect})
\begin{equation}
\label{result}
    \varphi=\frac{2}{n}\text{d}\ln{f}.
\end{equation}
Assume for now that the Weyl $1$-form $\varphi$ is exact, that is $\varphi=\text{d}\Omega$ for some function $\Omega$ on $M.$ Then from Equation (\ref{result}), the Weyl prior is given by
\begin{equation}
\label{weylprior}
^W\omega=\exp\{\frac{n}{2}\Omega\}\sqrt{\det{g}}.
\end{equation}
By comparison of Equation (\ref{weylprior}) with Proposition \ref{alpha}, the theorem is proved under the assumption of the exactness of the Weyl $1$-form.

Since we proved that the Weyl prior $^W\omega$ is the $\alpha-$prior with $\alpha=-n,$ and we required the existence of $^n\omega,$ our assumption of the exactness of the Weyl $1-$form $\varphi$ is indeed true by Remark \ref{comparison}.
\end{proof}
\begin{remark}
\label{minus}
The minus sign in $\alpha=-n$ is a result of the definition of $\alpha$-connection. By Remark \ref{dual}, the dual connection of $^\alpha\nabla$ is $^{-\alpha}\nabla,$ we would have $\alpha=n$ here, had we defined the $\alpha$-connection to be its dual connection in Definition \ref{def}. It would seem more natural to consider the dual prior of the Weyl prior.
\end{remark}
\section{Weyl Prior for Gaussian Family}
In this section, we calculate the Weyl prior of the Gaussian family as an example.
\begin{example}[Gaussian Family]
Consider the Gaussian family
\end{example}
$$
\mathcal{P}=\left\{p(x|\mu,\sigma^2)=\frac{1}{\sqrt{2\pi}\sigma}\exp{\left[-\frac{1}{2\sigma^2}(x-\mu)^2\right]}\ |\ (\mu,\sigma)\in{M}\right\}.
$$
Choose $(\mu,\sigma^2)$ as a coordinate system, we have
\begin{equation*}
    \partial_{\mu}l=\frac{x-\mu}{\sigma^2},
\end{equation*}
\begin{equation*}
    \partial_{\sigma^2}l=\frac{(x-\mu)^2}{2\sigma^4}-\frac{1}{2\sigma^2},
\end{equation*}
where $l$ is the log likelihood function.

The first element of the Fisher metric tensor $g$ in $(\mu,\sigma^2)$-coordinate is given by
\begin{align*}
    g_{\mu\mu}&=\text{E}_{\theta}\left[\partial_{\mu}l\partial_{\mu}l\right]\\
    &=\int_{-\infty}^\infty\frac{(x-\mu)^2}{\sigma^4}\frac{1}{\sqrt{2\pi}\sigma}\exp{\left[-\frac{1}{2\sigma^2}(x-\mu)^2\right]}dx\\
    &=\frac{1}{\sigma^2},
\end{align*}
where $E_\theta$ is the conditional expectation of $X$ given $\Theta.$
The other elements of the Fisher metric tensor are
\begin{equation*}
    g_{\mu\sigma^2}=g_{\sigma^2\mu}=0,
\end{equation*}
and
\begin{equation*}
    g_{\sigma^2\sigma^2}=\frac{1}{2\sigma^4}.
\end{equation*}
Hence,
\begin{equation}
    \sqrt{\det{g}}=\frac{1}{\sqrt{2}\sigma^3}.
\end{equation}
To calculate the Weyl $1$-form, we first calculate the Amari-Chentsov tensor $C,$
\begin{align*}
    C_{\mu\mu\mu}&=\text{E}_{\theta}\left[\partial_{\mu}l\partial_{\mu}l\partial_{\mu}l\right]\\
    &=\int_{-\infty}^\infty\frac{(x-\mu)^3}{\sigma^6}\frac{1}{\sqrt{2\pi}\sigma}\exp{\left[-\frac{1}{2\sigma^2}(x-\mu)^2\right]}dx\\
    &=0.
\end{align*}
Similarly,
\[
    C_{\sigma^2\mu\mu}=C_{\mu\sigma^2\mu}=C_{\mu\mu\sigma^2}=\frac{1}{\sigma^4},
\]
\[
C_{\sigma^2\sigma^2\mu}=C_{\sigma^2\mu\sigma^2}=C_{\mu\sigma^2\sigma^2}=0,
\]
and
\[
    C_{\sigma^2\sigma^2\sigma^2}=\frac{1}{\sigma^6}.
\]
Hence, the Weyl $1$-form is given by
\begin{equation}
\begin{split}
    \varphi &= \frac{1}{2}C_{ijk}g^{jk}d\theta^i\\
    & =\frac{3}{2\sigma^2}d\sigma^2.
\end{split}
\end{equation}
Now, it is easy to check that $\varphi = \text{d}(\frac{3}{2}\ln{\sigma^2})$ is an exact form. Hence for Gaussian family $\mathcal{P},$ Weyl prior exists and is given by
\begin{equation}
    \begin{split}
        ^W\omega&=\exp{\left\{\frac{2}{2}\frac{3}{2}\ln{\sigma^2}\right\}}\frac{1}{\sqrt{2}\sigma^3}\\
        &=\frac{1}{\sqrt{2}}.
    \end{split}
\end{equation}

\begin{remark}
Based on our calculation, we find that the Weyl prior for the univariate Gaussian distribution with unknown mean and unknown variance is just the uniform prior. This shows that the uniform prior is in fact an uninformative prior. This counter-intuitive result is related to the fact that every 2-dimensional manifold is conformally-flat, which can be proved using the existence of isothermal coordinates in 2 dimensions \cite{KULKARNI}.
\end{remark}

\section{Multivariate Gaussian}
The above example can be extended to the multivariate case. Consider the multivariate Gaussian distribution
\begin{equation}
    f(x|\mu,\Sigma)=\frac{1}{(2\pi)^{n/2}\sqrt{\det{\boldsymbol{\Sigma}}}}\exp{\{-\frac{1}{2}(\boldsymbol{x}-\boldsymbol{\mu})^{\intercal}\boldsymbol\Sigma^{-1}(\mathbf{x}-\boldsymbol{\mu})\}},
\end{equation}
where $\boldsymbol\mu$ is the mean vector and $\boldsymbol\Sigma$ is the covariance matrix.

Using matrix calculus, we have
\begin{equation}
    \partial_{\boldsymbol\mu}l=\boldsymbol\Sigma^{-1}(\boldsymbol{x}-\boldsymbol\mu)
    \end{equation}
and
\begin{equation}
    \begin{split}
        \partial_{\boldsymbol\Sigma}l &= -\frac{1}{2}\boldsymbol\Sigma^{-1}+\frac{1}{2}\boldsymbol\Sigma^{-1}(\boldsymbol{x}-\boldsymbol\mu)(\boldsymbol{x}-\boldsymbol\mu)^{\intercal}\boldsymbol\Sigma^{-1}\\
        &=-\frac{1}{2}\boldsymbol\Sigma^{-1}+\frac{1}{2}[\boldsymbol\Sigma^{-1}(\boldsymbol{x}-\boldsymbol\mu)]\otimes[\boldsymbol\Sigma^{-1}(\boldsymbol{x}-\boldsymbol\mu)]\\
        &=-\frac{1}{2}\boldsymbol\Sigma^{-1}+\frac{1}{2}(\boldsymbol\Sigma^{-1}\otimes\boldsymbol\Sigma^{-1})[(\boldsymbol{x}-\boldsymbol\mu)\otimes(\boldsymbol{x}-\boldsymbol\mu)]
    \end{split}
\end{equation}

We can now compute the Fisher information matrix.
\begin{equation}
    \begin{split}
         g_{\boldsymbol\mu\boldsymbol\mu}&=\text{E}_{\theta}[\partial_{\boldsymbol\mu}\otimes\partial_{\boldsymbol\mu}]\\
         &=(\boldsymbol\Sigma^{-1}\otimes\boldsymbol\Sigma^{-1})\text{E}_{\theta}[(\boldsymbol{x}-\boldsymbol\mu)\otimes(\boldsymbol{x}-\boldsymbol\mu)]\\
         &=(\boldsymbol\Sigma^{-1}\otimes\boldsymbol\Sigma^{-1})\boldsymbol\Sigma\\
         &=\boldsymbol\Sigma^{-1}\boldsymbol\Sigma\boldsymbol\Sigma^{-\intercal}\\
         &=\boldsymbol\Sigma^{-1}
    \end{split}
\end{equation}
where the second last line is by the definition of covariance matrix.

Similarly,
\begin{equation}
    g_{\boldsymbol\mu\boldsymbol\Sigma}=g_{\boldsymbol\Sigma\boldsymbol\mu}=\boldsymbol{0},
\end{equation}
and
\begin{equation}
    g_{\boldsymbol\Sigma\boldsymbol\Sigma} = \frac{1}{2}\boldsymbol\Sigma^{-1}\otimes\boldsymbol\Sigma^{-1}.
\end{equation}
The Amari-Chentsov tensor can be computed in the same way.
\begin{equation}
    \begin{split}
        C_{\boldsymbol\mu\boldsymbol\mu\boldsymbol\mu}=\text{E}_{\theta}[\partial_{\boldsymbol\mu}l\otimes\partial_{\boldsymbol\mu}l\otimes\partial_{\boldsymbol\mu}l]=\boldsymbol{0}.
    \end{split}
\end{equation}
\begin{equation}
    C_{\boldsymbol\Sigma\boldsymbol\mu\boldsymbol\mu}=C_{\boldsymbol\mu\boldsymbol\Sigma\boldsymbol\mu}=C_{\boldsymbol\mu\boldsymbol\mu\boldsymbol\Sigma}=\boldsymbol\Sigma^{-1}\otimes\boldsymbol\Sigma^{-1}.
\end{equation}
\begin{equation}
        C_{\boldsymbol\mu\boldsymbol\Sigma\boldsymbol\Sigma}=C_{\boldsymbol\Sigma\boldsymbol\mu\boldsymbol\Sigma}=C_{\boldsymbol\Sigma\boldsymbol\Sigma\boldsymbol\mu}=\boldsymbol{0}.
\end{equation}
\begin{equation}
\label{last}
    C_{\boldsymbol\Sigma\boldsymbol\Sigma\boldsymbol\Sigma}=\boldsymbol\Sigma^{-1}\otimes\boldsymbol\Sigma^{-1}\otimes\boldsymbol\Sigma^{-1}.
\end{equation}
The detail computation of Eqn (\ref{last}) is as follows.
\begin{equation}
    \begin{split}
        C_{\boldsymbol\Sigma\boldsymbol\Sigma\boldsymbol\Sigma}&=\text{E}_{\theta}[\partial_{\boldsymbol\Sigma}l\otimes\partial_{\boldsymbol\Sigma}l\otimes\partial_{\boldsymbol\Sigma}l]\\
        &=\text{E}_{\theta}\{-\frac{1}{8}\boldsymbol\Sigma^{-1}\otimes\boldsymbol\Sigma^{-1}\otimes\boldsymbol\Sigma^{-1}+\frac{1}{8}\boldsymbol\Sigma^{-1}\otimes(\boldsymbol\Sigma^{-1}\otimes\boldsymbol\Sigma^{-1})\left[(\boldsymbol{x}-\boldsymbol{\mu})\otimes(\boldsymbol{x}-\boldsymbol{\mu})\right]\otimes\boldsymbol\Sigma^{-1}\\
        &+\frac{1}{8}\left(\boldsymbol\Sigma^{-1}\otimes\boldsymbol\Sigma^{-1}\right)\left[(\boldsymbol{x}-\boldsymbol{\mu})\otimes(\boldsymbol{x}-\boldsymbol{\mu})\right]\otimes(\boldsymbol\Sigma^{-1}\otimes\boldsymbol\Sigma^{-1})-\frac{1}{8}\left(\boldsymbol\Sigma^{-1}\otimes\boldsymbol\Sigma^{-1}\otimes\boldsymbol\Sigma^{-1}\otimes\boldsymbol\Sigma^{-1}\right)\\
        &(\boldsymbol{x}-\boldsymbol{\mu})\otimes(\boldsymbol{x}-\boldsymbol{\mu})\otimes(\boldsymbol{x}-\boldsymbol{\mu})\otimes(\boldsymbol{x}-\boldsymbol{\mu})\otimes\boldsymbol{\Sigma}^{-1}+\frac{1}{8}\boldsymbol{\Sigma}^{-1}\otimes\boldsymbol{\Sigma}^{-1}\otimes\left(\boldsymbol{\Sigma}^{-1}\otimes\boldsymbol{\Sigma}^{-1}\right)\\
        &\left[(\boldsymbol{x}-\boldsymbol{\mu})\otimes(\boldsymbol{x}-\boldsymbol{\mu})\right]-\frac{1}{8}\boldsymbol{\Sigma}^{-1}\otimes\left(\boldsymbol{\Sigma}^{-1}\otimes\boldsymbol{\Sigma}^{-1}\right)\left[(\boldsymbol{x}-\boldsymbol{\mu})\otimes(\boldsymbol{x}-\boldsymbol{\mu})\right]\otimes\left(\boldsymbol{\Sigma}^{-1}\otimes\boldsymbol{\Sigma}^{-1}\right)\\
        &\left[(\boldsymbol{x}-\boldsymbol{\mu})\otimes(\boldsymbol{x}-\boldsymbol{\mu})\right]-\frac{1}{8}\left(\boldsymbol\Sigma^{-1}\otimes\boldsymbol\Sigma^{-1}\right)\left[(\boldsymbol{x}-\boldsymbol{\mu})\otimes(\boldsymbol{x}-\boldsymbol{\mu})\right]\otimes\boldsymbol{\Sigma}^{-1}\otimes\left(\boldsymbol{\Sigma}^{-1}\otimes\boldsymbol{\Sigma}^{-1}\right)\\
        &\left[(\boldsymbol{x}-\boldsymbol{\mu})\otimes(\boldsymbol{x}-\boldsymbol{\mu})\right]+\frac{1}{8}\left(\boldsymbol\Sigma^{-1}\otimes\boldsymbol\Sigma^{-1}\otimes\boldsymbol\Sigma^{-1}\otimes\boldsymbol\Sigma^{-1}\otimes\boldsymbol\Sigma^{-1}\otimes\boldsymbol\Sigma^{-1}\right)[(\boldsymbol{x}-\boldsymbol{\mu})\otimes(\boldsymbol{x}-\boldsymbol{\mu})\\
       &\otimes(\boldsymbol{x}-\boldsymbol{\mu})\otimes(\boldsymbol{x}-\boldsymbol{\mu})\otimes(\boldsymbol{x}-\boldsymbol{\mu})\otimes(\boldsymbol{x}-\boldsymbol{\mu})]\}=\boldsymbol\Sigma^{-1}\otimes\boldsymbol\Sigma^{-1}\otimes\boldsymbol\Sigma^{-1}\\
    \end{split}
\end{equation}
The above expression can be evaluated by 4-th and 6-th moments of multivariate Gaussian.

The Weyl prior is then given by:
\begin{equation}
    \begin{split}
        \varphi&=\frac{1}{2}C_{ijk}g^{jk}d\theta^{i}\\
        &=\frac{1}{2}C_{\boldsymbol\Sigma\boldsymbol{\mu}\boldsymbol{\mu}}g^{\boldsymbol{\mu}\boldsymbol{\mu}}d\boldsymbol\Sigma+\frac{1}{2}C_{\boldsymbol\Sigma\boldsymbol\Sigma\boldsymbol\Sigma}g^{\boldsymbol\Sigma\boldsymbol\Sigma}d\boldsymbol\Sigma\\
        &=\frac{1}{2}\left(\boldsymbol\Sigma^{-1}\otimes\boldsymbol\Sigma^{-1}\right)\boldsymbol\Sigma d\boldsymbol\Sigma+\frac{1}{2}\left(\boldsymbol\Sigma^{-1}\otimes\boldsymbol\Sigma^{-1}\otimes\boldsymbol\Sigma^{-1}\right)2\left(\boldsymbol\Sigma\otimes\boldsymbol\Sigma\right)d\boldsymbol\Sigma\\
        &=\frac{3}{2}\boldsymbol\Sigma^{-1}d\boldsymbol\Sigma\\
        &=\text{d}\left(\frac{3}{2}\ln{\det{\boldsymbol{\Sigma}}}\right).
    \end{split}
\end{equation}
The Weyl prior is thus given by:
\begin{equation}
\label{result2}
    \begin{split}
        ^W\omega&=\exp{\left\{\frac{n+(n+1)n/2}{2}\frac{3}{2}\ln\det{\boldsymbol{\Sigma}}\right\}}\sqrt{\det{\boldsymbol{\Sigma}^{-1}}\det{\left[\frac{1}{2}\boldsymbol\Sigma^{-1}\otimes\boldsymbol\Sigma^{-1}\right]}}\\
        &=\left(\det{\boldsymbol\Sigma}\right)^{(3n^2+9n)/8}\frac{\left(\det{\boldsymbol{\Sigma}}\right)^{(-2n-1)/2}}{2^{n/2}}\\
        &=\frac{\left(\det{\boldsymbol{\Sigma}}\right)^{(n-1)(3n+4)/8}}{2^{n/2}},
    \end{split}
\end{equation}
where in the first line, $n+(n+1)n/2$ is the dimension of the statistical manifold for the multivariate Gaussian distribution.
\begin{remark}
Our calculation of the Weyl prior of the multivariate Gaussian distribution is generally not an uniform prior. But Eqn (\ref{result2}) shows that when $n=1$, that is, the univariate case, the Weyl prior is indeed the uniform prior. This is in accord with our direct calculation for the univariate case.
\end{remark}

\section{Conclusion and discussion}
We discussed Weyl geometry and Weyl prior in this paper. We also calculated Weyl prior for the Gaussian family as an example.

The underlying principle of Jeffreys prior, $\alpha$-parallel prior, and Weyl prior is the concept of invariance in statistics. Jeffreys prior is invariant under a change of the coordinate of parameters. Weyl prior and $\alpha$-parallel prior, as generalizations of Jeffreys prior, automatically satisfy this invariance. Moreover, Weyl prior, as a volume form defined on a Weyl manifold, is also invariant under a gauge transformation \cite{calinarticle}. Also invariant under the gauge transformation is the generalized conjugate connection \cite{calinarticle}.

One possible use of the Weyl prior is using the uniform prior for distributions with 2 parameters. This is because any 2-dimensional manifold is conformally-flat.

\end{document}